\documentclass[11pt,reqno]{amsart}

\usepackage{mathrsfs}
\usepackage{multicol}

\usepackage{amssymb,latexsym,amsbsy,amsthm,amsxtra,amscd,amsfonts,amsthm,amsmath,verbatim}

\usepackage[all,cmtip]{xy}
\usepackage{lmodern}
\usepackage[T1]{fontenc}

\usepackage[margin=1in]{geometry}
\usepackage{paralist}

\input xy
\xyoption{all}

\usepackage{mathrsfs}
\usepackage{multicol}
\usepackage{fancyvrb}
\usepackage{color}
\usepackage{amssymb,latexsym,amsbsy,amsthm,amsxtra,amscd,amsfonts,amsthm,amsmath,verbatim}

\makeatletter

\@addtoreset{equation}{section}
\def\theequation{\thesection.\@arabic \c@equation}

\makeatother

\usepackage{multicol}
\usepackage{fancyvrb}

\def\theequation{\arabic{equation}}
\def\theequation{\thesection.\arabic{equation}}
\numberwithin{equation}{section}

\def\charac{\operatorname{char}}
\def\deg{\operatorname{deg}}

\def\ker{\operatorname{ker}}

\def\reg{\operatorname{reg}}

\newcommand{\kk}{\Bbbk}

\def\lrar{{\longrightarrow}}

\def\A{\mathbb A}
\def\P{\mathbb P}

\newcommand{\Mod}[1]{\ (\mathrm{mod}\ #1)}
\newcommand{\ncom}{\newcommand}
\ncom{\bq}{\begin{equation}}
\ncom{\eq}{\end{equation}}
\ncom{\beqn}{\begin{eqnarray*}}
\ncom{\eeqn}{\end{eqnarray*}}
\ncom{\beq}{\begin{eqnarray}}
\ncom{\eeq}{\end{eqnarray}}
\ncom{\been}{\begin{enumerate}}
\ncom{\eeen}{\end{enumerate}}
\ncom{\olin}{\overline}
\ncom{\f}{\frac}
\ncom{\rar}{\rightarrow}

\def\nno{\nonumber}

\newcommand{\p}{\mathfrak p}

\newcommand{\m}{\mathfrak m}

\theoremstyle{plain}
\newtheorem{theorem}[equation]{Theorem}
\newtheorem{corollary}[equation]{Corollary}
\newtheorem{proposition}[equation]{Proposition}
\newtheorem{lemma}[equation]{Lemma}

\theoremstyle{definition}
\newtheorem{notation}[equation]{Notation}

\newtheorem{remark}[equation]{Remark}

\ncom{\bib}{\bibitem}

\ncom{\maxi}{\underset{\underset{}{i}}{\max}}

\ncom{\limm}{\underset{{ n \to \infty}}{\lim}}

\ncom{\mprime}{\m^{\prime}}

\begin{document}
 \title[ ] {Resurgence and Castelnuovo-Mumford regularity of certain monomial curves in $\A^3$ (accepted for publication in AMV)}
 \author{Clare D'Cruz}
 \address{Chennai Mathematical Institute, Plot H1 SIPCOT IT Park, Siruseri, 
Kelambakkam 603103, Tamil Nadu, 
India
} 
\email{clare@cmi.ac.in} 

 \keywords{Resurgence, Waldschmidt Constant,  Regularity}
 \thanks{The author  was partially funded by a grant from Infosys Foundation}

 \subjclass[2010]{Primary: 13A30, 1305, 13H15, 13P10} 
 
\begin{abstract}
Let $\p$ be the defining ideal of the monomial curve  ${\mathcal C}(2q+1, 2q+1+m, 2q+1+2m)$ in  the affine space $\A_k^3$ parameterized by $(x^{2q +1}, x^{2q +1 + m}, x^{2q +1 +2 m})$ where $gcd( 2q+1,m)=1$. In this paper we compute the resurgence of $\p$, the  Waldschmidt constant  of $\p$ and the Castelnuovo-Mumford regularity of the symbolic powers of $\p$.
\end{abstract}

\maketitle
\large
\section{Introduction}
 Let    $R= \kk[x_1, x_2, x_3]$ and  $S = \kk[x]$ be a polynomial rings over a field $\kk$ of characteristic zero. 
 Let $q$ and $m$ be positive integers, $d=2q+1$ and   $\gcd (d,m)=1$.  Consider the  homomorphism  $\phi: R \lrar S$  defined by $\phi(x_i) = x^{ d + (i-1)m}$, where   $1 \leq i \leq 3$. Throughout this paper   $\p:=\p_{{\mathcal C}(d, d+m, d+2m)}= \ker(\phi)$. For $q=1$, the resurgence $\rho(\p)$, the Waldschmidt constant $\gamma(\p)$ and the Castelnuovo-Mumford regularity of the symbolic powers of $\p$  have been computed in \cite{clare-shreedevi}. In this paper we generalise these results for all $q \geq 1$. 
We also verify that certain conjectures posed in   \cite{harb-huneke} hold true for $\p$.
Before we describe our main results we will  give some background on these quantities. 
 
 For any ideal $I$ in a Noetherian ring  $A$ of positive dimension with no embedded components, the $n$-th symbolic power of $I$ is defined by
 $I^{(n)}:= \cap_{\p \in Ass(R/I)} I^n A_{\p} \cap A$. In general, the generators of $I^{(n)}$ are hard to describe.
Hence,  in order  to have a more precise   relation between symbolic powers and ordinary powers of ideals, Harbourne 
posed the following conjecture: Let $I \subseteq \kk[x_1, \ldots, x_t]$ be an homogeneous ideal. Then 
$I^{(m)}\subseteq  I^r$ if $m \geq  r(t-1) - (t-2)$ \cite[Conjecture~8.4.2]{brauer}. In the same paper, the authors give  evidence  to show that  this conjecture is true if $\charac \kk >0$. 
  Later, Bocci 
and Harbourne  introduced an asymptotic quantity called resurgence which is defined as $\rho(I) :=\sup \{  m/r |   
I^{(m)}  \not \subset I^r\}$ \cite{bocci-harbourne}. 
This 
supremum exists and in fact $1 \leq \rho(I) \leq t-1$  \cite[Lemma~2.3.2]{bocci-harbourne}. 
Since resurgence in general is hard to compute, in  \cite{bocci-harbourne}  the authors  define another invariant which they call the Waldschmidt 
 constant.  The Waldschmidt constant was first introduced by Waldschmidt  in \cite{waldschmidt}. We use the definition as in \cite{bocci-harbourne}. 
Let $\alpha(I):= \min \{ n | I_n \not = 0 \}$. The Waldschmidt   constant is defined as 
$
\gamma(I) = \limm~\f{\alpha( I^{(n)})}{n}.
$
Bocci and Harbourne showed that if $I$ is a homogenous ideal, then  $\alpha(I) / \gamma(I) \leq \rho(I)$,  and in addition if $I$ is a  zero 
dimensional subscheme  in a projective space, then
$\alpha(I) / \gamma(I) \leq \rho(I) \leq \reg (I) / \gamma(I)$, where $\reg(I)$ id the Castelnuovo-Mumford 
regularity of $I$ \cite[Theorem~1.2.1]{bocci-harbourne} .  

The resurgence  and the Waldschmidt constant has been studied in a few cases: for certain general points in $
\P^2$  \cite{bocci-harbourne-2},  smooth subschemes  
\cite{guardo}, fat linear subspaces  \cite{fatabbi},  special point configurations \cite{duminicki} and monomial ideals 
\cite{bocco-waldschmidt}.   
  
  If we put   weights on the variables $wt(x_i) = d  + (i-1)m$ for  $i=1, 2, 3$, then   from \cite[Theorem~6.8]{clare} it follows that  $\p^{(n)}$ is  a   weighted homogenous  ideal of 
height $2$.   Hence, we can define the the Waldschmidt constant $\gamma(\p)$ in the same way as in  \cite{bocci-harbourne}. 
From  \cite[Theorem~1.1]{cut-kurano}  it follows that $\limm  \reg(    (\p^{n})^{sat}   ) /n $ exists and can even be irrational \cite{cutkosky}. Moreover,
 $\reg(\p^{(n)})$  is eventually periodic  \cite[Corollary~4.9]{cut-kurano}. 
In our case $ (\p^{n})^{sat} = \p^{(n)}$. We compute $ \reg(    (\p^{n})^{sat} )$ (Theorem~\ref{final regularity}). It is clear from our result that the regularity depends on $q$ and $m$.  
 In this paper we compute the exact formula for the resurgence of $\p$  (Theorem~\ref{theorem resurgence}).

We briefly summarise the contents of this paper. In Section~2 we prove some preliminary results. In Section~3 we compute the resurgence of $\p$. We  verify that   Conjecture~2.1 and Conjecture~4.1.5 in \cite{harb-huneke} hold true for $\p$ (Corollary~\ref{conjecture 2.1}, Corollary~\ref{conjecture 4.1.5}). In Section~4 we compute the Waldschmidt constant. 
We  verify that Chudnovsky's conjecture (Proposition~\ref{chudnovsky}) holds true in our case.  In section~5 we compute the Castelnuovo-Mumford regularity of $\p^{(n)}$ for all $n \geq 1$ (Theorem~\ref{final regularity}). 

We end this paper by observing that   Theorem~1.2.1 of \cite{bocci-harbourne} holds true for $\p$, i.e.,  $\alpha(\p) / \gamma(\p) \leq \rho(\p) \leq \reg (\p) / \gamma(\p)$.

\section{Preliminaries}
In this section we prove some results which may be well known. 
\begin{lemma}
\label{f_2}
For all $q \geq 1$ and $m \geq 1$,
\been
\item
\label{generators of p}
$\p = (g_1, g_2, g_3)$ where 
\beq
\label{def of delta}
 g_1 &:=& x_1^{m+q}x_2 - x_3^{q+1}, \hspace{.2in}
  g_2  := x_1^{m+q+1} - x_2 x_3^q, \hspace{.2in}
 g_3 := x_2^2 - x_1 x_3.
\eeq
\item 
\label{definition of f_i}
Let
 $
f := - x_1^{2(m+q)+1}- x_1^{m+q-1} x_2^3 x_3^{q-1}+3 x_1^{m+q} x_2 x_3^{q}-x_3^{2q+1}.$
\been
\item
\label{x_if_2}
For all  $i=1,2,3$, $x_i f \in \p^2$.
\item
\label{f_2^i} 
For all $j=1, \ldots, q+1$, $f^j \in \p^{2j-1}$.
\item
 \label{Eqn:f_2} 
  $\p^{(2)} = \p^2 + (f)$ and  for all $k \geq 1$, 
\beq
        \p^{(2k)}
&=&  (\p^{(2)})^k        \hspace{.2in} \mbox{and}
          \hspace{.2in}
           \p^{(2k + 1)} =\p \p^{( 2k)}\hspace{.1in}.
                    \eeq  
\eeen
\eeen
\end{lemma}
\begin{proof} \eqref{generators of p} is well known.

(\ref{x_if_2}) Since
\beqn
x_1 f = -g_2^2                                        -    x_3^{q-1}g_1 g_3, \hspace{.2in} 
x_2 f = - x_1^{m+q-1} x_3^{q-1} g_3^2 -  g_1g_2,  \hspace{.2in} 
x_3 f =   - g_1^2                                     + x_1^{m+q-1} g_2 g_3
\eeqn
and $g_j \in \p$ for all $j=1,2,3$,  we get $x_i f \in \p^2$  for all $i=1,2,3$. 

\eqref{f_2^i}  Let $1\leq j \leq q+1$. As $f = x_3^q g_1 - x_1^{m+q} g_2+x_1^{m+q-1}x_2x_3^{q-1} g_3$, 
 \beqn
            f^{j} 
   & =& (x_3^q g_1 - x_1^{m+q} g_2+x_1^{m+q-1}x_2x_3^{q-1} g_3)f^{j-1}\\
   & =& (x_3 f)^{j-1} x_3^{q-j+1}  g_1  -  (x_1 f)^{j - 1} x_1^{m+q-j+1}  g_2   + ( x_1f)^{j-1}   x_1^{m+q-j}x_2  g_3  \\
   & \in & \p^{2(j-1)} \p     \hspace{4.2in} \mbox{[from \eqref{x_if_2}]}  \\
   &=&  \p^{2j-1}.
  \eeqn
  
(\ref{Eqn:f_2}) follows from  \cite[Theorem~5.9]{clare} and   by induction on $k$.
\end{proof}

\section{Computation of resurgence} 
In this section we compute the resurgence $\rho(\p)$. The resurgence can be computed in the following way.
Let $\rho_n(\p):=\min\{r: \p^{(n)} \nsubseteq \p^r\}.$ Then
 $$
 {\displaystyle 
 \rho(\p):=\sup\left\{\frac{n}{\rho_n(\p)}:  n \geq 1\right\}.}
 $$
We state Conjecture~4.1.1 in \cite{harb-huneke} in our context: 
 Does $\p^{(2n-1)} \subseteq \p^n$ hold true for all $n$? The following proposition proves a stronger statement. 
 
\begin{proposition}
\label{Prop:Contain}
 Let  $k \geq 0$.  Then 
 \beqn
 \rho_{k (2q+2)  + j}(\p)  
 =\begin{cases} 
 k(2q+1)  +j + 1& \mbox{ if } k \geq 1 \mbox{ and } j=0,1\\   
   k(2q+1) +j  & \mbox{ if } k \geq 0 \mbox{ and }  j=2, \ldots, 2q+1  
 \end{cases}.
 \eeqn
\end{proposition}
\begin{proof} 
We first show that 
\beq
\label{equation-contain}
 \p^{(k (2q+2)  + j)}
 \subseteq 
 \begin{cases} 
 \p^{k(2q+1)  +j} & \mbox{ if } k \geq 1 \mbox{ and } j=0,1\\   
  \p^{ k(2q+1) +j-1}  & \mbox{ if } k \geq 0 \mbox{ and }  j=2, \ldots, 2q+1  
 \end{cases}.
 \eeq
Applying  Lemma~\ref{f_2},  (\ref{f_2^i}) and   (\ref{Eqn:f_2})  we get             
\begin{align}
 \label{containment of even symb powers}
           \p^{(2 j^{\prime})} 
&=  (\p^2+(f))^{ j^{\prime}}
=      (\sum_{i=0}^{ j^{\prime}} f^i \p^{2( j^{\prime}-i)} )
         \subseteq  \p^{2 j^{\prime}} + \sum_{i=1}^{ j^{\prime}} \p^{2i - 1} \p^{2( j^{\prime}-i)} 
        \subseteq \p^{2 j^{\prime} -1} ,&& j^{\prime}  = 1, \ldots, q+1
         \\
       \label{containment of  odd symb powers}
         \p^{(2 j^{\prime}   + 1)}
&= \p  \p^{(2 j^{\prime} )} \subseteq \p \p^{2 j^{\prime}-1} =  \p^{(2 j^{\prime})} 
       \hspace{2.0in}   \mbox{[by   (\ref{containment of even symb powers})]},  &&j^{\prime}  = 1, \ldots, q. 
 \end{align}
 Hence  (\ref{equation-contain}) is true for $k=0$ and $j = 2, \ldots, 2q+1$.

Let $k \geq 1$ and $j=0,1$. Then from Lemma~\ref{f_2}(\ref{Eqn:f_2}) and  (\ref{containment of even symb powers})  we get
\beq
\label{containment of 2q+2}
 \p^{(k (2q+2) + j)} 
= ( \p^{( 2 q+2)})^k  \p^{j}
\subseteq    \p^{  k(2q+1) + j   }  .
\eeq
Let $k \geq 1$ and  $j = 2, \ldots, 2q+1$.  Then from  Lemma~\ref{f_2}(\ref{Eqn:f_2}), (\ref{containment of even symb powers}), (\ref{containment of  odd symb powers}) and  (\ref{containment of 2q+2})  
 we get
 \beqn
 \p^{(k (2q+2) + j)} 
 =  \p^{(k (2q+2))} \p^{(j)} 
 \subseteq  \p^{ k (2q+1)}  \p^{j-1}
=\p^{k (2q+1)   + j-1}  .
\eeqn
To complete the proof of the lemma it remains to show that 
\beq
\label{Eqn:non-Contain}
{\p}^{k (2q+2)  + j}
 \not \subseteq \begin{cases} 
 \p^{k(2q+1)  +j + 1}& \mbox{ if } k \geq 1 \mbox{ and } j=0,1\\   
  \p^{ k(2q+1) +j } & \mbox{ if } k \geq 0 \mbox{ and }  j=2, \ldots, 2q+1  
 \end{cases}.
 \eeq
By  Lemma~\ref{f_2},   (\ref{generators of p}) and (\ref{Eqn:f_2}), $g_1 \in \p$ and  $f \in \p^{(2)}$. Hence  
\begin{align*}
          f^{k(q+1)} 
&\in \p^{(k(2q+2))},  & k &\geq 1 \\
         g_1 f^{k(q+1)}
&\in  \p \p^{(k(2q+2))}  = \p^{(k(2q+2) + 1)}, & k &\geq 1 \\
          f^{k(q+1)  +j^{\prime}}  
&\in \p^{(2k(q+1) + 2j^{\prime} )},  & k &\geq 0,  j^{\prime}=1, \ldots, q\\
            g_1 f^{k(q+1)  +j^{\prime}}  
&\in \p \p^{(2k(q+1) + 2j^{\prime} )} =  \p^{(2k(q+1) + 2j^{\prime}  +1)},& k &\geq 0,  j^{\prime}=1, \ldots, q. 
\end{align*}
From  (\ref{def of delta}) and Lemma~\ref{f_2}(\ref{definition of f_i}),
\begin{align}
 \label{equation-cong-f}
 f &\equiv x_3^{2q+1} \Mod {x_1}\\
  \label{equation-cong-g}
g_1 &\equiv x_3^{q+1}\Mod{x_1} \\
 \label{equation-cong-p}
\p &\equiv  (x_2^2, x_2x_3^q, x_3^{q+1})\Mod{x_1}.
\end{align}
By (\ref{equation-cong-f}), 
(\ref{equation-cong-g}) and (\ref{equation-cong-p}) we get
\begin{align*}
                 f^{k(q+1)}  
&\equiv  ( x_3^{q+1})^{k (2q+1) }  
\not \in      \p^{ k(2q+1) + 1} \Mod{x_1}, &  k \geq 1\\
                 g_1 f^{k(q+1)}  
&\equiv  (x_3^{q+1})^{(k (2q+1) +1)  } 
\not \in  \p^{ k(2q+1) + 2}\Mod{x_1}, & k \geq 1.
\end{align*}
As $   (2q+1)  (k (q+1)  + j^{\prime}  ) - (q+1) (k(2q+1) + 2j^{\prime} )= -j^{\prime} <0$, by (\ref{equation-cong-f}), 
(\ref{equation-cong-g}) and (\ref{equation-cong-p}) we get
\begin{align*}
               f^{k(q+1)  +j^{\prime}}  
&\equiv  x_3^{   (2q+1)  (k (q+1) + j^{\prime} )} 
 \not \in  \p^{ k(2q+1) + 2j^{\prime}}  \Mod{x_1}, 
&            k \geq 0, j^{\prime}=1, \ldots, q\\
               g_1 f^{k(q+1)  +j^{\prime}} 
&\equiv  x_3^{   (2q+1)  (k (q+1) + j^{\prime} ) + (q+1) } 
 \not  \in  \p^{2k(q+1) + 2j^{\prime}  +1} \Mod{x_1}, 
&            k \geq 0, j^{\prime}=1, \ldots, q.
\end{align*}
\end{proof}

We are now ready to compute the resurgence. 
\begin{theorem}
\label{theorem resurgence}
For all $q \geq 1$, 
 $\rho(\p)= \frac{2q+2}{2q+1} .$
\end{theorem}
\begin{proof}
By Proposition \ref{Prop:Contain}
\[
 \rho(\p)=\sup_k\left\{   \frac{k(2q+2) }{k(2q+1) +1},     \frac{k(2q+ 2)  +1}{k(2q+1) +2},   \frac{k(2q+2) + j }{k(2q+1) +j}  : j=2, \ldots, 2q+1\right\}
 =\frac{2q+2}{2q+1} 
 .
\]
\end{proof}

The following conjecture was stated for ideal of fat points  \cite[Conjecture~2.1]{harb-huneke}. We verify that the conjecture holds true for $\p$.

\begin{corollary}
\label{conjecture 2.1} For all $n \geq 1$,
${\displaystyle
 \p^{(2n)} \subseteq 
 \begin{cases}
 \m^n \p^n & \mbox{ if } q=1 \\
 \m^{2n} \p^n  & \mbox { if } q>1
 \end{cases}
}$
and 
${\displaystyle
 \p^{(2n-1)} \subseteq 
 \begin{cases}
 \m^n \p^n & \mbox{ if } q=1 \\
 \m^{2n} \p^n  & \mbox { if } q>1
 \end{cases}.
}$
\end{corollary}
\begin{proof} By Lemma~\ref{f_2}(\ref{Eqn:f_2}),  $\p^{(2n)} = (\p^{(2)})^n$. Hence    it is enough to prove the lemma for $n=1$. If $n=1$, then by Lemma~\ref{f_2}(\ref{Eqn:f_2}),
 \beqn
 \p^{(2)} = \p^2 + (f) = (g_1, g_2, g_3)\p + (x_3^q g_1 - x_1^{m+q} g_2+x_1^{m+q-1}x_2x_3^{q-1} g_3)
 \subseteq 
 \begin{cases}
 \m \p & \mbox { if } q=1\\
\m^2 \p &\mbox { if }q>1
 \end{cases}.
 \eeqn
 Hence the corollary is true for even powers.
By Lemma~\ref{f_2}(\ref{Eqn:f_2}),  $\p^{(2n-1)} = \p \p^{(2(n-1))}$. Hence  the corollary is true for odd powers.
\end{proof}

 We rephrase Conjecture~4.1.5   of \cite{harb-huneke} in our context:
  \begin{corollary}
 \label{conjecture 4.1.5}
 For all $n \geq 1$, $\p^{(2n-1)} \subseteq \m^{n-1} \p^{n}$. 
 \end{corollary}
 \begin{proof}  The proof follows from Corollary~\ref{conjecture 2.1}.
 \end{proof}

\section{Waldschmidt Constant}
Put weights $d_i= wt(x_i)$ where $d_1  = d := 2q+1$, $d_2 = d+m = 2q+1+m$ and $d_3 = d+2m = 2q+1  + 2m$. With these weights, $\p^n$ and  $\p^{(n)}$ are  weighted homogenous ideals \cite{clare}. 
Hence we can define  $\alpha(\p):= \min \{ n | \p_n \not = 0 \}$. The Waldschmidt  
constant can be defined as 
$$
\gamma(\p) = \limm~\f{\alpha( \p^{(n)})}{n}.
$$
In this section we compute $\alpha(\p)$ and  $\gamma(\p)$.

\begin{theorem}
\label{theorem gamma}
For all $q \geq 1$ and $m \geq 1$,
\been
\item
\label{theorem gamma 1}
$\alpha(\p) =  2 d_2$.

\item
\label{theorem gamma 2}
$
{\displaystyle
\gamma(\p) 
= \begin{cases}
15 / 2  & \mbox{ if } q=1 \mbox{ and } m=1\\
2 d_2& \mbox{ otherwise } .
\end{cases}
}
$
\eeen
\end{theorem}
\begin{proof} 
By Lemma~\ref{f_2}, $\p = (g_1, g_2, g_3)$, $\p^{(2n)} = (\p^{2} + f)^n$ and $\p^{(2n+1)} = \p \p^{(2n)}$
 where $g_1$, $g_2$, $g_3$ are  defined in (\ref{def of delta}) and $f$ is defined in Lemma~\ref{f_2}(\ref{definition of f_i}).  Hence,  
 $\deg(g_1) = (q+1) d_3 =(q+1)(d+2m)$, 
 $\deg(g_2) = d(m+q+1)$ and $\deg(g_3) = 2 d_2 = 2(d+m)$. 
This gives $\deg(g_3) \leq \deg(g_2) \leq \deg(g_1)$. Hence,
$\alpha(  \p^{(2n)}) = \min \{   2\deg ( g_3) n,  \deg(f) n\}$.

As $\deg(f) = d( d+2m)$, we get  $\deg(f) - 2 \deg(g_3) = d(d+2m) - 4(d+m) = d(d-4) + 2m (d-2)$,
we get, 
 $deg(f) \leq 2  \deg(g_3)$ if and only if $q=1$ and $m=1$. Hence 
\beq
\label{waldschmidt even}
\alpha( \p^{(2n)} )&= &
\begin{cases}
 \deg(f) n  = 15n  &   \hspace{.8in}\mbox{ if } q=1 \mbox{ and } m=1\\
2 \deg( g_3) n =  2(2n) d_2 & \hspace{.8in}\mbox { otherwise }
\end{cases} \\
\alpha( \p^{(2n + 1)}) \label{waldschmidt odd}
&= &
\begin{cases}
 n \deg(f) + \deg(g_3)    = 15n + 8&   \mbox{ if }  q=1 \mbox{ and } m=1\\
(2n+1)  \deg( g_3) =    2(2n+1)d_2       & \mbox { otherwise } 
\end{cases}.
\eeq
This implies that 
$
{\displaystyle
\gamma(\p) =  \limm~\f{\alpha( \p^{(n)})}{n} 
= \begin{cases}
15 / 2  &   \mbox{ if }  q=1 \mbox{ and } m=1\\
 2 d_2& \mbox { otherwise } .
\end{cases}
}$
 \end{proof}

We verify Chudnovsky's conjecture (see \cite[Remark~3.4]{harb-huneke}).
\begin{proposition}
\label{chudnovsky}
$
{\displaystyle
\f{\alpha  ( \p^{(n)})}{n} \geq \f{\alpha( \p) + 1}{2}.}
$
\end{proposition}
\begin{proof}
If $q=1$ and $m=1$, then by Theorem~\ref{theorem gamma}(\ref{theorem gamma 1}), $\f{\alpha(\p)  + 1}{2}= (8+1)/2  = 9/2$ and 
\beqn
\f{\alpha  ( \p^{(n)})}{n} 
= \begin{cases}
\f{(15/2)2r}{2r}  = \f{15}{2} \geq \f{9}{2}& \mbox{ if $n=2r$ }\\
\f{ (15(2r+1)  + 1)/2} {2r+1}  = \f{15}{2} + \f{1}{2(2r+1)} \geq \f{9}{2}& \mbox { if $n=2r+1$ } 
\end{cases}.
\eeqn
If either $q \not =1$ or $m \not = 1$, then by (\ref{waldschmidt even}), (\ref{waldschmidt odd}) and Theorem~\ref{theorem gamma}(\ref{theorem gamma 1}), for all $n \geq 1$
\beqn
\f{\alpha  ( \p^{(n)})}{n}  
= 2d_2 \geq \f{ 2 d_2 + 1}{2} = \f{\alpha(\p) + 1}{2}.
\eeqn
\end{proof}

\section{Regularity}
Recall $d=d_1  = 2q+1$, $d_2 = d+m$ and 
$d_3 = 2q+1 + 2m$. 
We begin with some basic results comparing $\p^{(n)}$ and $I_nT$ where $T = \kk[x_2, x_3] \cong R / (x_1)$.  
Let
\beq
\label{definition of Ji}
J_1 &:= &   \{x_2^2, x_2x_3^q, x_3^{q+1}\}, \hspace{.2in}
J_2 := \{x_3^{d}\}.
\eeq. 
\begin{notation}
If $A_1, \ldots, A_n$  are $n$ sets of monomials we define the set $A_1 \cdots A_n$ by
$A_1\cdots A_n := \{a_1 \cdots a_n : a_i \in A_i \}$.
\end{notation}
With the above notation
\beq
\label{equation of In}
I_n
&:=&  
\sum_{a_1 + 2 a_2 = n} J_{1}^{a_1} J_{2}^{a_{2} }.
\eeq
As an immediate consequence of Theorem~5.9 in \cite{clare} we have:
\begin{lemma}
\label{ideals in tprime}
For all $n \geq 1$,  $\p^{(n)}R + (x_1) = I_nR + (x_1)$,  $\p^{(n)} T = I_n T$,  $I_{2n} T = (I_2T)^n $ and $I_{2n+1} T = (I_2T) (I_{2n} T)$. 
\end{lemma}

 \begin{lemma}
 \label{reg comparision}
For all $n \geq 1$,  $\reg( R/ \p^{(n)}   )  =  \reg(T/ I_nT) $. 
  \end{lemma}
 \begin{proof} As $x_1$ is a nonzerodivisor on $R/ \p^{(n)}$ and $T / I_n T$, 
  \beqn
  \reg \left(  \f{R}  { \p^{(n)}} \right)
&=& \reg \left(  \f{R  }{ \p^{(n)} + (x_1)}  \right) - (d_1-1)  \hspace{.6in}   \mbox{\cite[Remark~4.1]{chardin}}\\
&=& \reg    \left(  \f{R }  { I_{n}R + (x_1) }  \right)  - (d_1-1)   \hspace{.6in} \mbox{[Lemma~\ref{ideals in tprime}]}\\
&=& \reg  \left(  \f{R } { I_{n}    R }   \right) +  (d_1-1) -  (d_1-1)     \hspace{.4in} \mbox{\cite[Remark~4.1]{chardin}} \\
&=& \reg  \left(  \f{T} { I_{n}   T  }   \right)  . 
  \eeqn
\end{proof}

From Lemma~\ref{reg comparision} it follows that  we need to compute $\reg(T/ I_nT)$.
\begin{corollary}
\label{regularity modulo x_2^2}
Let  $n \geq 1$.  Then 
\beqn
\reg\left( \f{T}{I_{n}T + ( x_2^2) } \right) 
= 
\begin{cases}
\f{ d   d_3}{2}n + 2 d_2 -2& \mbox{ if } n = 2r,\\
\f{d  d_3}{2} n + d_2 + \left(\f{  -d}{2} + q+1\right) d_3 -2& \mbox{ if }  n = 2r-1    . 
\end{cases}
\eeqn
\end{corollary}
\begin{proof}
If  $n=2r$,   then by Lemma~\ref{ideals in tprime}, $I_{2r}T + (x_2^2)= ( x_2^2, x_3^{  dr })$ and  hence
\beqn
\reg \left( \f{T }{ I_{2r}T + (x_2^2)} \right) = 2    d_2 + d d_3 r-2 = \f{ d  d_3}{2}n + 2 d_2 -2. 
\eeqn
If $n= 2r-1$,  then by Lemma~\ref{ideals in tprime},
$I_{2r-1}T + (x_2^2) 
=  (I_{2(r-1)  }T)(I_1T) + (x_2^2) =  (x_3^{d(r-1)}) (x_2 x_3^q, x_3^{q+1}) + (x_2^2) 
= (x_2^2, x_2x_3^{ d(r-1) + q },   x_3^{  d(r-1) + q +1} )$.
 By Hilbert-Burch theorem the minimal free resolution  of $I_{2r-1} T+ (x_2^2)$ is of the form
$$
\small
\xymatrix@C=10pt{
0\ar[r]  &  {\begin{array}{c}  T[ - 2d_2 - (d(r-1) + q)d_3 ] \\ \oplus  \\T[ -d_2  - (d( r-1) + q+ 1) d_3]\end{array}}
\ar[rrrrr]^(0.5){  \left( 
\begin{array}{cc}
x_3^{  d(r-1) + q } & 0\\
-x_2           & -x_3\\
0                & x_2
\end{array} \right)}
&&&&& { \begin{array}{c} T[-(2d_2)]  \\ \oplus \\ T[ - (d_2 + ( d(r-1) + q)d_3) ]  \\  \oplus \\ T[- ((d(r-1) + q+ 1) d_3)]\end{array}} 
\ar[r] & T
\ar[r] &  {\displaystyle \f{T}{I_{2r-1} T+ (x_2^2)}}
\ar[r] & 0.
}
$$
\normalsize
Hence
\beqn
\reg \left(\f{T}{I_{2r-1} T+ (x_2^2)} \right) 
=  d_2 + ( d(r-1) + q+1)d_3   -2
=  \f{d  d_3}{2} n + d_2 + \left(\f{  -d}{2} + q+1\right) d_3 -2 . 
\eeqn
\end{proof}

\begin{lemma}
\label{lemma mod x_3}
For all $n \geq 1$, 
\beqn
\reg\left( \f{T}{I_{2n} T+ ( x_3^d) } \right) 
= 2d_2(2n) -2d_2+ d  d_3-2.
\eeqn
\end{lemma}
\begin{proof}
By  Lemma~\ref{ideals in tprime} we get
\beqn
    I_{2n} T+ (x_3^d) 
= I_2^n T+ (x_3^d) 
= (x_2^4, x_2^3x_3^q, x_2^2 x_3^{q+1}, x_3^d)^{n} + (x_3^d)
= (x_2^{4n}, x_2^{4n-1}x_3^{q}, x_2^{4n-2} x_3^{q+1}, x_3^d).
\eeqn
By Hilbert-Burch theorem the minimal free resolution of $ I_{2n}T + (x_3^d)$ is 
\small
$$
\xymatrix@C=12pt{
0\ar[r]  &  {\begin{array}{c}  
T[ - (4n-1 )d_2 -(q+1) d_3  ] \\ \oplus  \\T[ -4n d_2 - qd_3]
\\ \oplus \\
T[-(4n-2) d_2 - d  d_3]
\end{array}}
\ar[rrrrr]^(0.5){  \left( 
\begin{array}{ccc}
0 & x_3^q      &0\\
x_3          & -x_2 & 0\\
 -x_2 & 0 &- x_3^q\\
 0& 0& x_2^{4n-2}
\end{array} \right)}
&&&&& { \begin{array}{c} T[-4n d_2]  \\ \oplus \\ T[ -(4n-1) d_2  - qd_3]  \\  \oplus \\ T[- (4n-2)  d_2 - (q+1) d_3]\\
\oplus \\T[ -d  d_3]
\end{array}} 
\ar[r] & T
\ar[r] &  {\displaystyle \f{T}{I_{2n}T + (x_3^d)}}
\ar[r] & 0.
}
$$
\normalsize
This gives $\reg( T/ I_{2n} T+ (x_3^d)) =  (4n-2) d_2 + d d_3 - 2 = 2d_2(2n) -2d_2+ d  d_3-2
$. 
\end{proof}

\begin{proposition}
\label{regularity even}
Let  $n \geq 1$. Then
\beqn
  \reg \left( \f{T}{I_{2n}T }\right) 
&=& \begin{cases} 
(2d_2)(2n) - 2d_2 + d d_3 -2  & \mbox{ if } q=1 \mbox{ and } m=1,\\
 \f{d d_3}{2} (2n) + 2d_2-2& \mbox{ otherwise }    . 
\end{cases}.
\eeqn
\end{proposition}
\begin{proof} For all $n \geq 1$, 
\beq \nno
(I_{2n} T : x_3^d) \nno
&=& \sum_{a_1 + 2 a_2 = 2n} ((J_1T)^{a_1} ((J_2T)^{a_2} : x_3^d)\\ \nno
&=& \left( \sum_{a_2=1}^{n} ((J_1T)^{a_1} (J_2T)^{a_2} : x_3^d)\right) + (J_1T)^{2n} : x_3^d)\\ \nno
&\subseteq &  \left( \sum_{a_2=1}^{n} (  (J_1T)^{a_1} (J_2T)^{a_2-1} : x_3^d) \right) + J_1^{2n-2}  
\hspace{.2in} \mbox{[\cite[(3.4)]{clare}]}\\
&\subseteq& I_{2n-2}T.
\eeq
The other inclusion follows from the fact that  $x_3^d I_{2n-2}T \subseteq (J_2T) (I_{2n-2}T) = I_{2n}T $.
Hence we have the exact sequence
$$
\xymatrix@C=20pt{
0 \ar[r] &{\displaystyle \f{T}{I_{2n-2}T } [-d d_3]}
\ar[r]^(0.7){.x_3^d}  & {\displaystyle \f{T}{I_{2n} T}}
\ar[r] & {\displaystyle \f{T}{I_{2n} T+ (x_3^d)}}
\ar[r] & 0\\
}
$$
This implies that 
\beqn
 &&  \reg\left( \f{T}{I_{2n}T}\right)\\
&=& \max\left\{    \reg\left( \f{T}{I_{2n-2} T}\right) + d  d_3 , 
                            \reg\left( \f{T}{I_{2n} T  + (x_3^d) }\right)  \right\}\\
&=& \max\left\{    \reg\left( \f{T}{I_{2n-4} T}\right) + 2d d_3 , 
                           \reg\left( \f{T}{I_{2n-2}  T + (x_3^d) }  \right) + d   d_3 , 
                           \reg\left( \f{T}{I_{2n}T   + (x_3^d) }\right)   \right\}\\
&=& \vdots\\
&=& \max\left\{   \reg\left( \left. \f{T}{I_{2n-2i} T+ (x_3^d)}\right) + di d_3 \right| 
                                                                                      i=0, \ldots, n-1  \right\}\\
&=& \max\left\{ \left. 2d_2(2n-2i)   -2d_2 + d  d_3  -2         + di  d_3\right| i=0, \ldots, n-1  \right\} 
        \hspace{1in} 
        \mbox{[by Lemma~\ref{lemma mod x_3}]}\\
&=& \begin{cases} 
(2d_2)(2n) - 2d_2 + d  d_3 -2  & \mbox{ if } q=1 \mbox{ and } m=1,\\
2d_2(2)   -2d_2 + d  d_3  -2         + d (n-1)  d_3 & \mbox{ otherwise}   
\end{cases}\\
&=& \begin{cases} 
(2d_2)(2n) - 2d_2 + d  d_3 -2  & \mbox{ if } q=1 \mbox{ and } m=1,\\
 \f{d d_3}{2} (2n) + 2d_2-2  & \mbox{ otherwise}   
\end{cases}.
\eeqn
\end{proof}

\begin{proposition}
\label{regularity odd}
 Let  $n \geq 1$. Then
\beqn
  \reg \left( \f{T}{I_{2n+1} T}\right) 
= \begin{cases}
  (2d_2) (2n+1)   -2 d_2 + d   d_3 -2 + 2d_2 & \mbox{ if }  q=1 \mbox{ and } m=1\\
\f{3d_3}{2} (2n+1) + 4d_2 -\f{d  d_3}{2} -2  & \mbox{ if } q=1 \mbox{ and }  m=2 \\
    \f{d  d_3}{2}  (2n+1) +d_2 + \left(\f{  -d}{2} + q+1\right) d_3-2     & \mbox{ if }   q=1  \mbox{ and }  m \geq 3 \mbox{ or } q  \geq 2
\end{cases}.
\eeqn
\end{proposition}
\begin{proof}
 For all $n \geq 1$, the sequence
\normalsize
$$
\xymatrix@C=20pt{
0 \ar[r] &{\displaystyle \f{T}{I_{2n}T}[-2d_2] }
\ar[r]^{.x_2^2}  & {\displaystyle \f{T}{I_{2n+1} T}}
\ar[r] & {\displaystyle \f{T}{I_{2n+1}T + (x_2^2)}}
\ar[r] & 0\\
}
$$
is exact by \cite[Theorem~3.1]{clare}. Hence
\beqn
\label{regularity odd and even}
 &&    \reg\left( \f{T}{I_{2n+1} T}\right) \\
&= &  \max\left\{    \reg\left( \f{T}{I_{2n} T}\right) + 2 d_2 , 
         \reg\left( \f{T}{I_{2n+1}   T+ (x_2^2) }\right)  \right\}\\
&=&  \begin{cases}
         \max\left\{  (2d_2) (2n)   -2 d_2 + d   d_3 -2 + 2 d_2,
         {d d_3} n +d_2 + ( q+1) d_3 -2    \right\}
&       \mbox{ if }  q=1 \mbox{ and } m=1\\
         \max\left\{ d d_3 (n) + 2d_2  -2 + 2 d_2,
         {d  d_3}  (n) +d_2 + ( q+1) d_3-2     \right\}   
&       \mbox{ otherwise } 
         \end{cases} \\
&&    \hspace{4in} \mbox{[Proposition~\ref{regularity even},  Corollary~\ref{regularity modulo x_2^2}]}\\
&=&  \begin{cases}
        (2d_2) (2n+1)   -2 d_2 + d d_3 -2 
&     \mbox{ if }  q=1 \mbox{ and } m=1\\
       \f{d  d_3}{2} (2n+1) + 4d_2 -\f{d d_3}{2} -2 
&     \mbox{ if } q=1 \mbox{ and }  m=2 \\
       \f{d d_3}{2}  (2n+1) +d_2 + \left(\f{  -d}{2} + q+1\right) d_3-2     
&    \mbox{ if }   q=1  \mbox{ and }  m \geq 3 \mbox{  or } q  \geq 2
\end{cases}.
\eeqn  
\end{proof}

\begin{theorem}
\label{final regularity}
\been
\item
\label{final regularity one}
 $\reg (R/ \p)= d_2 + (q+1) d_3-2$. 

\item 
\label{final regularity two}
Let $n \geq 2$. 
\been
\item If $q=1$ and $m=1$, then 
$
{\displaystyle \reg (R/ \p^{(n)}) 
= (2d_2)n -2d_2 + d  d_3-2}$.

\item
If $q=1$ and  $m=2$, then  ${
\displaystyle 
  \reg \left( \f{R}{ \p^{(n) } } \right) 
= \begin{cases}
   \f{d  d_3}{2}n + 4d_2 -\f{d  d_3}{2} -2& \mbox{ if  $n$ is odd},\\
   \f{d  d_3}{2}n   + 2d_2-2  & \mbox{ if  $n$ is even}   . 
\end{cases}
}
$

\item
If $q=1$ and $m \geq 3$ or $q \geq 2$, then  $${\displaystyle 
  \reg \left( \f{R}{ \p^{(n) } } \right) 
= \begin{cases}
    \f{d  d_3}{2}  n +d_2 + \left(\f{  -d}{2} + q+1\right) d_3-2    & \mbox{ if  $n$ is odd}\\
  \f{d  d_3}{2}n   + 2d_2-2 & \mbox{ if  $n$ is even}  . 
\end{cases}
}$$
\eeen
\eeen
\end{theorem}
\begin{proof}
By Lemma~\ref{reg comparision},   $\reg (R/ \p^{(n)} )= \reg (T/  I_nT) $. 
Hence  (\ref{final regularity one})  follows from Corollary~\ref{regularity modulo x_2^2} and 
(\ref{final regularity two}) follows from Proposition~\ref{regularity even} and Proposition~\ref{regularity odd}. 
\end{proof}

We end this paper with the following remark.
\begin{remark}
From our computations one can verify that for all $q \geq 1$ and $m \geq 1$,
\beqn
\f{\alpha(\p)}{ \gamma(\p)} \leq \rho(\p) \leq \f{\reg(\p)}{\gamma(\p)}. 
\eeqn
It follows that  Theorem~1.2.1 of \cite{bocci-harbourne} holds true for $\p$.
\end{remark}

\end{document}